\documentclass[a4paper]{amsart}
\numberwithin{equation}{section}
\newtheorem{theorem}{Theorem}[section]
\newtheorem{corollary}{Corollary}[section]
\newtheorem{definition}{Definition}[section]
\newtheorem{lemma}{Lemma}[section]

\theoremstyle{remark}

\usepackage{hyperref}
\usepackage{color}
\usepackage{graphicx}

\title[Radii of convexity of integral operators]
 {Radii of convexity of integral operators}
\subjclass[2010]{30C45}
\keywords{Radii of Convexity; Starlike; Convex; Locally Convex; Integral Operators; Subordination.}

\begin{document}
\begin{abstract}
The object of the present paper is to study of radius of convexity two certain integral operators as follows
\begin{equation*}
  F(z):=\int_{0}^{z}\prod_{i=1}^{n}\left(f'_i(t)\right)^{\gamma_i}{\rm d}t
\end{equation*}
and
\begin{equation*}
  J(z):=\int_{0}^{z}\prod_{i=1}^{n}\left(f'_i(t)\right)^{\gamma_i}\prod_{j=1}^{m}
  \left(\frac{g_j(z)}{z}\right)^{\lambda_j}{\rm d}t,
\end{equation*}
where $\gamma_i, \lambda_i\in\mathbb{C}$, $f_i$ $(1\leq i\leq n)$ and $g_j$ $(1\leq j\leq m)$ belong to the certain subclass of analytic functions.
\end{abstract}
\author[P. Najmadi, Sh. Najafzadeh and A. Ebadian]
       {P. Najmadi, Sh. Najafzadeh and A. Ebadian}

\address{Department of Mathematics, Payame Noor University, Tehran, Iran}
       \email{najmadi@phd.pnu.ac.ir {\rm (P. Najmadi)}}
       \email{najafzadeh1234@yahoo.ie {\rm (Sh. Najafzadeh)}}
       \email{aebadian@pnu.ac.ir {\rm (A. Ebadian)}}
\maketitle
\section{Introduction}
\label{intro}

Let $\Delta:=\{z\in \mathbb{C}: |z|<1\}$ and let $\mathcal{H}(\Delta)$ be the set of all functions analytic in $\Delta$ and let
\begin{equation*}
  \mathcal{A}_p:=\{f\in \mathcal{H}(\Delta): f(z)=z^p+a_{p+1}z^{p+1}+\cdots\},
\end{equation*}
for all $z\in \Delta$ and $p\in \mathbb{N}:=\{1,2,\ldots\}$ with $\mathcal{A}_1\equiv \mathcal{A}$. Also let $\mathcal{S}$ be the subclass of $\mathcal{A}$ consisting of univalent functions.

If $f$ and $g$ are two of the functions in $\mathcal{A}$, we say
that $f$ is subordinate to $g$, written $f (z)\prec g(z)$, if there
exists a Schwarz function $w$ such that $f (z) = g(w(z))$, for all
$z\in\Delta$.
Furthermore, if the function $g$ is univalent in $\Delta$, then
we have the following equivalence:
\begin{equation*}
    f (z)\prec g(z) \Leftrightarrow (f (0) = g(0)\quad {\rm and}\quad f (\Delta)\subset g(\Delta)).
\end{equation*}

Let $\mathcal{S}^*(\varphi)$ denote the class of functions $f$ in $\mathcal{S}$
satisfying
\begin{equation*}
  \frac{zf'(z)}{f(z)}\prec \varphi(z)\qquad (z\in \Delta),
\end{equation*}
where $\varphi$ is an analytic function with positive real part on $\Delta$, $\varphi(0)=1$,
$\varphi'(0) > 0$ and $\varphi$ maps $\Delta$ onto a region starlike with respect to 1 and symmetric
with respect to the real axis. The class $\mathcal{S}^*(\varphi)$ was introduced by Ma and Minda (see \cite{maminda}). We note that the class $\mathcal{S}^*(\alpha)$ consisting of starlike functions of order $\alpha$ and the class $\mathcal{S}^*[A,B]$ of Janowski starlike functions are special cases of $\mathcal{S}^*(\varphi)$ where $\varphi(z)=(1+(1-2\xi)z)/(1-z)$ ($0\leq \xi<1$) and $\varphi(z)=(1+Az)/(1+Bz)$ ($-1\leq B<A\leq1$), respectively.

We also use the following well known notations:
\begin{equation*}
  \mathcal{S}^*_p(\xi):=\left\{f\in \mathcal{A}_p: \mathfrak{Re} \left(\frac{zf'(z)}{f(z)}\right)>\xi, \ \ 0\leq \xi<p, |z|<1\right\},
\end{equation*}
for the functions of p-valent starlike of order $\xi$ and
\begin{equation*}
  \mathcal{K}_p(\xi):=\left\{f\in \mathcal{A}_p: \mathfrak{Re} \left(1+\frac{zf''(z)}{f'(z)}\right)>\xi, \ \ 0\leq \xi<p, |z|<1\right\},
\end{equation*}
for the functions of p-valent convex of order $\xi$. We put $\mathcal{S}^*_1(\xi)\equiv \mathcal{S}^*(\xi)$ and $\mathcal{K}_1(\xi)\equiv \mathcal{K}(\xi)$, the class of starlike and convex functions of order $\xi$, respectively.

Following \cite{OPW}, for $\beta \in \mathbb{R}$ we consider the class
$\mathcal{G}\left(\beta\right)$ consisting of locally univalent functions
$f\in \mathcal{A}$ which satisfy the condition
\begin{equation}
\mathfrak{Re}\left( 1+\frac{zf^{\prime \prime }(z)}{f^{\prime }(z)}\right)
<1+\frac{\beta}{2} \qquad z\in \Delta.
\end{equation}
In \cite{SO}, Ozaki introduced the class $\mathcal{G}\equiv \mathcal{G}(1)$
and proved that functions in $\mathcal{G}$ are univalent in $\Delta $. In
\cite{TU}, Umezawa generalized Ozaki's result for a version of the class $\mathcal{G}$ (convex functions in one direction). It is also known that the
functions in the class $\mathcal{G}\left( 1\right) $ are starlike in $\Delta$. The function class $\mathcal{G}\left(\beta\right)$ was studied extensively by Kargar et al. \cite{KargarJMAA}.

For $a\in\Delta$, let $Aut(\Delta)$ be the set of all automorphisms $\phi(z)=e^{i\theta}\dfrac{z+a}{1+\bar{a}z}$, where $\theta\in\mathbb{R}$. Following \cite{Pom}, we recall a definition:
\begin{definition}\label{def M}
 A subclass $\mathcal{M}\subset\mathcal{A}$ is said to be a linear invariant family if:

(i) every $f\in\mathcal{M}$ is locally univalent in $\Delta$ and

(ii) for any $f\in\mathcal{M}$ and $\phi\in Aut(\Delta)$ we have
\begin{equation}\label{F}
F_{\phi}(f)(z)=\dfrac{f(\phi(z))-f(\phi(0))}{f^{\prime}(\phi(0))\phi^{\prime}(0)}\in\mathcal{M}.
\end{equation}
\end{definition}
Define also the order of $f\in\mathcal{A}$ as $ord$ $f=\sup_{\phi\in Aut(\Delta)}|a_{2}(F_{\phi})|$ and the universal linear invariant family of order $\alpha\geq1$ as
\begin{equation*}
\mathcal{U}_{\alpha}=\{f\in\mathcal{A}:~ord~f\leq\alpha\}.
\end{equation*}
It is well known that $\mathcal{U}_{1}=\mathcal{K}$, whereas $\mathcal{S}\subset\mathcal{U}_{2}$. The following lemma will be required.
\begin{lemma}\label{lem Ebadian, Kargar}{\rm (}see \cite{ebakar2017}{\rm )}
If $f\in\mathcal{U}_{\alpha}$ and $\alpha\geq1$, then
\begin{equation*}
\left|\dfrac{zf^{\prime\prime}(z)}{f^{\prime}(z)}-\dfrac{2|z|^{2}}{1-|z|^{2}}\right|
\leq\dfrac{2\alpha|z|}{1-|z|^{2}}\qquad(z\in\Delta).
\end{equation*}
\end{lemma}
Over the years, study on the integral operators have been investigated by many authors, including (for example) the following cases.

In \cite{silverman1975}, Silverman obtained the order of starlikeness of functions given by
\begin{equation*}
  z\prod_{i=1}^{m}\left(\frac{f_i(z)}{z}\right)^{a_i}\prod_{i=1}^{n}\left(g'_i(z)\right)^{b_i},
\end{equation*}
where $f\in \mathcal{S}^*(\alpha)$, $g\in \mathcal{K}(\alpha)$ and $a_i,b_i\geq0$. Also, Dimkov (see \cite{dimkov1991}) studied the operator
\begin{equation*}
  z\prod_{i=1}^{m}\left(\frac{f_i(z)}{z}\right)^{a_i}
\end{equation*}
and found the radii of starlikeness and convexity as well as orders of starlikeness and convexity. Again, Dimkov and Dziok \cite{dimdzi} considered the functions of the type
\begin{equation}\label{DimDzi}
  z^p\prod_{i=1}^{m}\left(\frac{f_i(z)}{z^p}\right)^{a_i},
\end{equation}
where $f_i\in \mathcal{S}^*_p(\alpha_i)$ and $a_i$ are the complex numbers. They found that the conditions for the centre and the radius of the disc $\{z\in \mathbb{C}: |z-w|<r\}$, contained in the unit disc $\Delta$ and containing the origin, so that its transformation by the function \eqref{DimDzi} be a domain starlike with respect to the origin.
In 2008, also Breaz et al. (see \cite{BOB}) introduced a new integral operator as follows:
\begin{equation}\label{Breaz IP}
  F(z):=\int_{0}^{z}\prod_{i=1}^{n}\left(f'_i(t)\right)^{\gamma_i}{\rm d}t
\end{equation}
and studied some properties of it where $\gamma_i$ are the complex numbers. For example, they showed that under certain conditions, the integral operator $F(z)$ is univalent, starlike and convex function.

In this paper we shall consider the following integral operator:
\begin{equation}\label{J}
  J(z):=\int_{0}^{z}\prod_{i=1}^{n}\left(f'_i(t)\right)^{\gamma_i}\prod_{j=1}^{m}
  \left(\frac{g_j(z)}{z}\right)^{\lambda_j}{\rm d}t,
\end{equation}
where $\gamma_i,\lambda_j\in \mathbb{C}$ and $f_i,g_j\in \mathcal{A}$. Note that by taking $g_j=f_j$ and $m=n$ in \eqref{J}, we have the integral operator $I^{\gamma_i,\lambda_i}_1$ that studied by Frasin (see \cite{frasin}).

In the next Section \ref{sec2} we obtain the radius of convexity of the integral operator \eqref{Breaz IP} where $f_i$ belonging to the classes $\mathcal{U}_\alpha$, $\mathcal{K}$, $\mathcal{S}$ and $\mathcal{G}(\beta)$. Moreover, the radii of convexity of the integral operator \eqref{J} is obtained where $f_i$ are the universal linear invariant, convex and locally convex functions and $g_j$ belong to the class of starlike functions of order $\alpha_j$.
\section{Main Results}\label{sec2}
 Our first result is contained in the following:
\begin{theorem}\label{t1}
  Let $f_i$ belong to the class $\mathcal{U}_{\alpha_i}$ where $\alpha_i\geq1$ $(1\leq i\leq n)$ and $\alpha:=\max\{\alpha_1,\ldots,\alpha_n\}$. Also, let $M>0$  and $\sum_{i=1}^{n}|\gamma_i|\leq M$ $(\gamma_i\in\mathbb{C})$. Then the radius of convexity of the integral operator $F(z)$ defined by \eqref{Breaz IP} is
  \begin{equation}\label{r}
    r_c(F)=\frac{\sqrt{\alpha^2 M^2+2M+1}-\alpha M}{2M+1}.
  \end{equation}
\end{theorem}
\begin{proof}
At first for a fixed $M>0$ we denote by $\mathfrak{F}(M)$, the class of all integral operators of the form \eqref{Breaz IP} such that
\begin{equation}\label{F(M)}
  \sum_{i=1}^{n}|\gamma_i|\leq M
\end{equation}
and we define the subclass $\mathfrak{F}(m,\gamma)$ as follows
\begin{equation*}
  \mathfrak{F}(m,\gamma):=\left\{F\in \mathfrak{F}(M): \sum_{i=1}^{n} \mathfrak{Re} \{\gamma_i\}=\gamma, \ \ -m\leq \gamma\leq m, 0<m<M\right\}.
\end{equation*}
So $\bigcup_{m\in(0,M]}\mathfrak{F}(m,\gamma)\subseteq\mathfrak{F}(M)$. Using the analytic definition of convexity, the radius $r$ of convexity of the class $\mathfrak{F}(M)$ is the largest number $0<r<1$ for which
\begin{equation*}
  \min_{|z|=r}\mathfrak{Re}\left\{1+\frac{zF''(z)}{F(z)}\right\}\geq0\qquad (F\in \mathfrak{F}(m,\gamma)).
\end{equation*}
By \eqref{Breaz IP} for every $r\in (0,1)$ we have
\begin{align*}
  \min_{|z|=r}\mathfrak{Re}\left\{1+\frac{zF''(z)}{F(z)}\right\} &= 1+\min_{|z|=r}\mathfrak{Re}\left\{\sum_{i=1}^{n}\gamma_i\left(\frac{zf''_i(z)}{f'_i(z)}\right)\right\} \\
  &=  1+\min_{|z|=r}\sum_{i=1}^{n}\mathfrak{Re}\left\{\gamma_i\left(\frac{zf''_i(z)}{f'_i(z)}\right)\right\} \\
  &\geq  1+\sum_{i=1}^{n}\min_{|z|=r}\mathfrak{Re}\left\{\gamma_i\left(\frac{zf''_i(z)}{f'_i(z)}\right)\right\}.
\end{align*}
From Lemma \ref{lem Ebadian, Kargar}, we get
\begin{align*}
\left|\gamma_i\left(\frac{zf^{\prime\prime}_i(z)}{f^{\prime}_i(z)}\right)
-\dfrac{2\gamma_i|z|^{2}}{1-|z|^{2}}\right|
&\leq\dfrac{2\alpha_i|z||\gamma_i|}{1-|z|^{2}}\\
&\leq \dfrac{2\alpha|z||\gamma_i|}{1-|z|^{2}}\qquad(z\in\Delta).
\end{align*}
Since an inequality $|a|\leq c$ implies $-c \leq \mathfrak{Re}\{a\} \leq c$, it follows from the last inequality that
\begin{equation*}
  \mathfrak{Re}\left\{\gamma_i\left(\frac{zf''_i(z)}{f'_i(z)}\right)\right\}\geq \frac{2r^2 \mathfrak{Re}\{\gamma_i\}}{1-r^2}-\frac{2\alpha r|\gamma_i|}{1-r^2}\quad(|z|=r<1).
\end{equation*}
Hence we have
\begin{equation}\label{psi}
  \min_{|z|=r}\mathfrak{Re}\left\{1+\frac{zF''(z)}{F(z)}\right\}\geq \frac{-(2M+1)r^2-2\alpha M r+1}{1-r^2}=:\psi(r).
\end{equation}
Therefore, $\psi(r)> 0$ if and only if $\omega(r):=-(2M+1)r^2-2\alpha Mr+1>0$. Since $\alpha^2 M^2+2M+1>0$, the zeros of $\omega(r)$ are real. It is easy to see that $\omega(r)>0$ when $0<|z|<r_c$, where
\begin{equation*}
  r_c=\frac{\sqrt{\alpha^2 M^2+2M+1}-\alpha M}{2M+1}
\end{equation*}
and $r_c\in(0,1)$. Also, $\omega(r)$ has a unique root in the interval $(0,1)$. This completes the proof.
\end{proof}
If we put $\alpha_i=1$ ($i=1,2,\ldots,n$) in the Theorem \ref{t1}, then we have.
\begin{corollary}
  Let $f_i$ {\rm (}$1\leq i\leq n${\rm )} belong to the class $\mathcal{K}$. Also let $\sum_{i=1}^{n}|\gamma_i|\leq M$. Then the radius of convexity of the integral operator \eqref{Breaz IP} is
  \begin{equation}\label{r}
    r'_c:=\frac{1}{2M+1}\qquad (M>0).
  \end{equation}
\end{corollary}
\begin{theorem}
  Let $f_i$ {\rm (}$1\leq i\leq n${\rm )} be univalent functions. Then the radius of convexity of the integral operator \eqref{Breaz IP} is
  \begin{equation}\label{r}
    r:=\frac{1}{\sqrt{4 M^2+2M+1}+2 M}\qquad (M>0).
  \end{equation}
\end{theorem}
\begin{proof}
From \cite{Ahlfors}, if $f\in \mathcal{S}$, then
\begin{equation*}
  \left|\frac{zf''(z)}{f'(z)}-\frac{2r^2}{1-r^2}\right|\leq \frac{4r}{1-r^2}\qquad(|z|=r<1).
\end{equation*}
The remain of the proof is similar to the proof of the Theorem \ref{t1} and we thus omit the details.
\end{proof}
\begin{theorem}
    Let $f_i$ belong to the class $\mathcal{G}(\beta_i)$, where $0<\beta_i \leq1$ and $1\leq i\leq n$. Also let $\beta:=\max\{\beta_1,\ldots,\beta_n\}$. Then the radius of convexity of the integral operator \eqref{Breaz IP} is
  \begin{equation}\label{r}
    r:=\frac{1}{\beta M+1}\qquad (M>0).
  \end{equation}
\end{theorem}
\begin{proof}
  From \cite[proof of Theorem 1]{OPW}, if $f_i$ belong to the class $\mathcal{G}(\beta_i)$, then we have
  \begin{align}\label{ineq, zf''}
    \left|\gamma_i\left(\frac{zf_i''(z)}{f_i'(z)}\right)\right|&\leq \frac{\beta_i|\gamma_i| |z|}{1-|z|}\nonumber\\
    &\leq\frac{\beta|\gamma_i| |z|}{1-|z|} .
  \end{align}
  Now by the proof of Theorem \ref{t1} and by using \eqref{ineq, zf''}, we obtain
  \begin{align*}
  \min_{|z|=r}\mathfrak{Re}\left\{1+\frac{zF''(z)}{F(z)}\right\} &\geq 1+\sum_{i=1}^{n}\min_{|z|=r}\mathfrak{Re}\left\{\gamma_i\left(\frac{zf''_i(z)}{f'_i(z)}\right)\right\}\\
  &\geq 1-\frac{\beta r}{1-r}\sum_{i=1}^{n}|\gamma_i|\\
  &\geq 1-\frac{\beta r}{1-r}M,
  \end{align*}
  where $|z|=r<1$ and $M>0$. The last inequality is non-negative if $0<r\leq1/(\beta M+1)$ and concluding the proof.
\end{proof}
In the next theorem, we obtain the radii of convexity of the integral operator \eqref{J} in special cases and at first we assume that $f_i\in\mathcal{U}_{\alpha_i}$ and $g_j\in\mathcal{S}^*(\xi_j)$.

\begin{theorem}\label{th. eshterak}
    Let $f_i$ belong to the class $\mathcal{U}_{\alpha_i}$ $(1\leq i\leq n)$, where $\alpha_i\geq1$, $\alpha:=\max\{\alpha_1,\ldots,\alpha_n\}$ and $g_j\in \mathcal{S}^*(\xi_j)$ $(1\leq j\leq m)$ where $0\leq \xi_j<1$ and $\xi=\max\{\xi_1,\ldots,\xi_m\}$. Also, let $\sum_{i=1}^{n}|\gamma_i|\leq M$ $(M>0, \gamma_i\in\mathbb{C})$ and $\sum_{j=1}^{m}|\lambda_j|\leq N$ $(N>0, \lambda_j\in\mathbb{C})$.
    Then the radius of convexity of the integral operator $J(z)$ defined by \eqref{J} is
  \begin{equation}\label{r}
    r_c(M,N):=\frac{\sqrt{[(\xi-1)N-\alpha M]^2-2[(\xi-1)N-M-1]}-(\xi-1)N+\alpha M}{2[(\xi-1)N-M-1]}.
  \end{equation}
\end{theorem}
\begin{proof}
  Let $J(z)$ be defined by \eqref{J}. It is easy to see that $J(z)\in \mathcal{A}$ and
  \begin{equation}\label{proof t2.3,1}
    1+\frac{zJ''(z)}{J'(z)}=1+\sum_{i=1}^{n}\gamma_i \frac{zf''_i(z)}{f'_i(z)}+\sum_{j=1}^{m}\lambda_j\left(\frac{zg'_i(z)}{g_i(z)}-1\right).
  \end{equation}
   We shall to show that
   \begin{equation*}
     \min_{|z|=r} \mathfrak{Re}\left\{1+\frac{zJ''(z)}{J'(z)}\right\}\geq 0.
   \end{equation*}
  From now \eqref{proof t2.3,1}, we obtain
  \begin{align*}
     \min_{|z|=r} \mathfrak{Re}\left\{1+\frac{zJ''(z)}{J'(z)}\right\} &\geq 1+\sum_{i=1}^{n}\min_{|z|=r}\mathfrak{Re}\left\{\gamma_i
     \left(\frac{zf''_i(z)}{f'_i(z)}\right)\right\} \\
    &+\sum_{j=1}^{m}\min_{|z|=r}\mathfrak{Re}\left\{\lambda_j
    \left(\frac{zg'_j(z)}{g_j(z)}-1\right)\right\}.
  \end{align*}
 Since $f_i\in \mathcal{U}_{\alpha_i}$, using the proof of Theorem \ref{t1}, we get
  \begin{equation}\label{ineq. proof f u}
  1+\sum_{i=1}^{n}\min_{|z|=r}\mathfrak{Re}\left\{\gamma_i
     \left(\frac{zf''_i(z)}{f'_i(z)}\right)\right\}\geq \frac{-(2M+1)r^2-2\alpha M r+1}{1-r^2}.
\end{equation}
Also, because $g_j\in \mathcal{S}^*(\xi_j)$ we have
\begin{equation*}
  \frac{zg'_j(z)}{g_j(z)}\prec \frac{1+(1-2\xi_j)z}{1-z}\qquad (z\in \Delta).
\end{equation*}
The subordination principle it follows that
\begin{align*}
  \left|\lambda_j\left(\frac{zg'_j(z)}{g_j(z)}-1\right)-\frac{2(1-\xi_j)\lambda_j r^2}{1-r^2}\right|&\leq \frac{2(1-\xi_j)|\lambda_j|r}{1-r^2}\\
  &\leq \frac{2(1-\xi)|\lambda_j|r}{1-r^2}\qquad(|z|=r<1).
\end{align*}
Thus
\begin{equation*}
  \min_{|z|=r}\mathfrak{Re}\left\{\lambda_j\left(\frac{zg'_j(z)}{g_j(z)}-1\right)\right\}
  \geq \frac{2(1-\xi)\mathfrak{Re}\{\lambda_j\} r^2}{1-r^2}-\frac{2(1-\xi)|\lambda_j|r}{1-r^2},
\end{equation*}
with the equality for the function $\frac{z}{(1-z)^{2(1-\xi
)}}$. We denote by $\mathfrak{G}(N)$, the class of all integral operators of the form \eqref{J} such that
\begin{equation}
  \sum_{j=1}^{m}|\lambda_j|\leq N\quad(N>0)
\end{equation}
and we denote the subclass $\mathfrak{G}(s,\lambda)$ as follows
\begin{equation*}
    \mathfrak{G}(s,\lambda):=\left\{J\in \mathfrak{G}(N): \sum_{j=1}^{m} \mathfrak{Re} \{\lambda_j\}=\lambda, \ \ -s\leq \lambda\leq s, 0<s<N\right\}.
\end{equation*}
Therefore $\mathfrak{G}(N)=\bigcup_{s\in (0,N]}\mathfrak{G}(s,\lambda)$. Hence we get
\begin{equation}\label{proof t2.3,2}
  \min_{|z|=r}\mathfrak{Re}\left\{\lambda_j\left(\frac{zg'_j(z)}
  {g_j(z)}-1\right)\right\}\geq
  \frac{-2(1-\xi)N (r+1)r}{1-r^2}\qquad(|z|=r<1).
\end{equation}
Now from \eqref{ineq. proof f u} and \eqref{proof t2.3,2}, we obtain
\begin{equation*}
  \min_{|z|=r} \mathfrak{Re}\left\{1+\frac{zJ''(z)}{J'(z)}\right\}\geq\frac{-2[(1-\xi)N+(M+1)]r^2
  -2(\alpha M+(1-\xi)N)r+1}{1-r^2}=:\varphi(r).
\end{equation*}
Easily seen that $\varphi(r)>0$ if $0<|z|=r_c(M,N)<1$, where
\begin{equation*}
  r_c(M,N)=\frac{\sqrt{[(\xi-1)N-\alpha M]^2-2[(\xi-1)N-M-1]}-(\xi-1)N+\alpha M}{2[(\xi-1)N-M-1]}
\end{equation*}
and concluding the proof.
\end{proof}
Taking $\alpha=1$ in the Theorem \ref{th. eshterak}, we get.
\begin{corollary}
      Let $f_i$ be convex functions for $1\leq i\leq n$ and $g_j\in \mathcal{S}^*(\xi_j)$ $(j=1,2,\ldots, m)$, $\xi=\max\{\xi_1,\ldots,\xi_m\}$. Then the radius of convexity of the integral operator $J(z)$ is
\begin{equation*}
  r_c(M,N)=\frac{\sqrt{[(\xi-1)N- M]^2-2[(\xi-1)N-M-1]}-(\xi-1)N+M}{2[(\xi-1)N-M-1]}\quad (M,N>0).
\end{equation*}
\end{corollary}

\begin{theorem}
    Let $f_i$ be locally convex univalent functions of order $\beta_i$, where $0<\beta\leq1$ and $1\leq i\leq n$ and $g_j\in \mathcal{S}^*(\xi_j)$ ($j=1,2,\ldots, m$), $\xi=\max\{\xi_1,\ldots,\xi_m\}$. Then the integral operatos $J(z)$ is convex in $|z|<r$, where $r$ is the positive root \begin{equation}\label{r}
   -[2(1-\xi)N+\beta M+1]r^2-[2(1-\xi)N+\beta M]r+1=0.
  \end{equation}
\end{theorem}
\begin{proof}
  From \eqref{ineq, zf''} and \eqref{proof t2.3,2}, we get
  \begin{equation*}
    \min_{|z|=r} \mathfrak{Re}\left\{1+\frac{zJ''(z)}{J'(z)}\right\}
    \geq\frac{-[2(1-\xi)N+\beta M+1]r^2-[2(1-\xi)N+\beta M]r+1}{1-r^2}=:\phi(r).
  \end{equation*}
  It is easy too see that $\phi(r)>0$ if the denominator of $\phi(r)>0$.
  The remain of proof is obvious and we omit the details. Thus the proof is completed.
\end{proof}

\end{document}